\UseRawInputEncoding
\documentclass[12pt]{article}

\usepackage{dsfont}
\usepackage{amsfonts,amsmath,amsthm}
\usepackage{algorithm, algorithmic}
\usepackage{color}
\usepackage{graphicx}
\usepackage{stmaryrd}        % provides \llbracket and \rrbracket

\usepackage[paper=a4paper,dvips,top=2cm,left=2cm,right=2cm,
    foot=1cm,bottom=4cm]{geometry}

\newcommand{\ve}{{\bf e}}

\begin{document}

\title{Eigenvalues and Singular Value Decomposition of Dual Complex Matrices}
\author{ Liqun Qi\footnote{%
    Department of Applied Mathematics, The Hong Kong Polytechnic University, Hung Hom,
    Kowloon, Hong Kong
    %Department of Mathematics, School of Science, Hangzhou Dianzi University, Hangzhou 310018 China
    ({\tt maqilq@polyu.edu.hk}).}
    \and and \
    Ziyan Luo\footnote{Department of Mathematics,
  Beijing Jiaotong University, Beijing 100044, China. (zyluo@bjtu.edu.cn). This author's work was supported by NSFC (Grant No.  11771038) and Beijing Natural Science Foundation (Grant No.  Z190002).}
}
\date{\today}
\maketitle

\begin{abstract}
We introduce right eigenvalues and subeigenvalues for square dual complex matrices.   An $n \times n$ dual complex Hermitian matrix has exactly $n$ right eigenvalues and subeigenvalues, which are all real.
The Hermitian matrix is positive semi-definite or definite if and only if all of its right eigenvalues and subeigenvalues are nonnegative or positive, respectively.  A Hermitian matrix can be diagonalized if and only if it has no right subeigenvalues.
Then we present the singular value decomposition
 %and an Eckart-Young like theorem
for general dual complex matrices.  The results are further extended to dual quaternion matrices.
%This work lays the foundation of dual complex matrix analysis, and paves the way for further study on dual quaternion matrices and their applications.

\medskip

  \medskip

  \textbf{Key words.} Dual complex matrices, Clifford algebra, right eigenvalues, right subeigenvalues, singular value decomposition %inpainting, color image denoising.

  %\medskip
  %\textbf{AMS subject classifications.}
\end{abstract}

\renewcommand{\Re}{\mathds{R}}
\newcommand{\rank}{\mathrm{rank}}
\renewcommand{\span}{\mathrm{span}}
\newcommand{\X}{\mathcal{X}}
\newcommand{\A}{\mathcal{A}}
\newcommand{\I}{\mathcal{I}}
\newcommand{\B}{\mathcal{B}}
\newcommand{\C}{\mathcal{C}}
\newcommand{\OO}{\mathcal{O}}
\newcommand{\e}{\mathbf{e}}
\newcommand{\0}{\mathbf{0}}
\newcommand{\dd}{\mathbf{d}}
\newcommand{\ii}{\mathbf{i}}
\newcommand{\jj}{\mathbf{j}}
\newcommand{\kk}{\mathbf{k}}
\newcommand{\va}{\mathbf{a}}
\newcommand{\vb}{\mathbf{b}}
\newcommand{\vc}{\mathbf{c}}
\newcommand{\vg}{\mathbf{g}}
\newcommand{\vr}{\mathbf{r}}
\newcommand{\vt}{\rm{vec}}
\newcommand{\vx}{\mathbf{x}}
\newcommand{\vy}{\mathbf{y}}
\newcommand{\vu}{\mathbf{u}}
\newcommand{\vv}{\mathbf{v}}
\newcommand{\y}{\mathbf{y}}
\newcommand{\vz}{\mathbf{z}}
\newcommand{\T}{\top}

\newtheorem{Thm}{Theorem}[section]
\newtheorem{Def}[Thm]{Definition}
\newtheorem{Ass}[Thm]{Assumption}
\newtheorem{Lem}[Thm]{Lemma}
\newtheorem{Prop}[Thm]{Proposition}
\newtheorem{Cor}[Thm]{Corollary}

\section{Introduction}

In 1873, W.K. Clifford \cite{Cl73} introduced dual numbers, dual complex numbers and dual quaternions.  These become the core knowledge of Clifford algebra or geometric algebra.

Dual numbers, dual complex numbers and dual quaternions now have wide applications in automatic differentiation, mechanics, geometry, rigid body motions, robotics and computer graphics \cite{BK20, BLH19, CKJC16, Da99, Gu11, MKO14, WYL12}.

On the other hand, quaternions found wide applications in color image and video processing \cite{CXZ19, JNS19, QLWZ21, XYXZN15}.  In these applications, quaternion matrices play a significant role.  The knowledge about right eigenvalues of quaternion has been developed for several decades \cite{Br51, QLWZ21, WLZZ18, Zh97}.   In particular, Zhang \cite{Zh97} presented singular value decomposition for quaternion matrices.  This work is widely cited in applications \cite{CXZ19, JNS19, QLWZ21, XYXZN15}.

Thus, we wonder if there are counter parts for dual number matrices, dual complex matrices and dual quaternion matrices.

Recently, Gutin \cite{Gu21} presented singular value decomposition for dual number matrices.  There is also study on dual complex matrices \cite{Br20}.  Thus, in this paper, we study dual complex matrices.

There are two difficulties in dual complex numbers.   They are non-communitative and there are zero divisors.   This makes dual complex matrix analysis more difficult, comparing with dual number matrix analysis and quaternion matrix analysis.

In the next section, we present some basic knowledge on dual complex numbers and dual complex matrices.

%In Section 3, we extend Johnson's theorem in 1944 for algebraic division rings to dual complex numbers.  %It can be extended to dual quaternions in the same line.

Then we define right eigenvalues for square dual complex matrices in Section 3. We show that the standard part of a right eigenvalue of a square dual complex matrix and the standard part of the corresponding right eigenvector are exactly a complex eigenvalue and its corresponding complex eigenvector of the standard part of the dual complex matrix.  This indicates that an $n \times n$ dual complex matrix has at most $n$ complex right eigenvalues.  We give a necessary and sufficient condition for a dual complex number to be a right eigenvalue of a square dual complex matrix.     %We then give an example that a square dual complex matrix has no complex right eigenvalue, but has dual complex right eigenvalues.

In Section 4, we first observe that the standard part of a dual complex Hermitian matrix is a complex Hermitian matrix, and the infinitesimal part is a complex skew symmetric matrix.   Then we show that the right eigenvalues of a dual complex Hermitian matrix are real.  Thus, an $n \times n$ dual complex Hermitian matrix has at most $n$ right eigenvalues.   In particular, a simple eigenvalue of the standard part of a dual complex matrix is a right eigenvalue of that dual complex Hermitian matrix.   An
$n \times n$ dual complex Hermitian matrix is diagonalizable if and only if it has $n$ right eigenvalues.

Then in Section 5, we define right subeigenvalues for square dual complex matrices.  There are two right subeigenvectors associated with a right subeigenvalue. These two right subeigenvectors are orthogonal to each other.  There is also a nonzero adjoint parameter associated with them.   If that adjoint parameter is degenerated to zero, then the right subeigenvalue was reduced to a right eigenvalue of multiplicity $2$.  We show that the right subeigenvalues of a dual complex Hermitian matrix are also real, and are multiple eigenvalues of the standard part of that dual complex Hermitian matrix.  Then we show that an $n \times n$ dual complex Hermitian matrix has exactly $n$ right eigenvalues and subeigenvalues.

In Section 6, we present the singular value decomposition
 %and an Eckart-Young like theorem
for general dual complex matrices.

Our work lays the foundation of dual complex matrix analysis, and paves the way for further study on dual quaternion matrices and their applications.   In Section 7, we briefly study extensions of our results to dual quaternion matrices.   Some final remarks are made in Section 8.

\section{Dual Complex Numbers and Dual Complex Matrices}

\subsection{Dual Complex Numbers}

We denote the real numbers, the complex numbers, the dual numbers and the dual complex numbers by ${\mathbb R}$, $\mathbb C$, ${\mathbb D}$ and $\mathbb {DC}$, respectively.  Scalars, vectors and matrices are denoted by small letters, bold small letters and capital letters, respectively.
A dual complex number $q$ has the form
$$q = q_0 + q_1\ii + q_2\epsilon\jj + q_3\epsilon\kk,$$
where $q_0, q_1, q_2$ and $q_3$ are real numbers, $\ii, \jj$ and $\kk$ are three imaginary units of quaternions, satisfying
$$\ii^2 = \jj^2 = \kk^2 =\ii\jj\kk = -1,$$
$$\ii\jj = -\jj\ii = \kk, \ \jj\kk = - \kk\jj = \ii, \kk\ii = -\ii\kk = \jj,$$
and $\epsilon$ is the infinitesimal unit, satisfying $\epsilon^2 = 0$.
These rules, along with the distribution law, determine the product of two dual complex numbers.   Unlike multiplication of dual numbers or of complex numbers, the multiplication of dual complex numbers is noncommutative.   Let $q = q_0 + q_1\ii + q_2\epsilon\jj + q_3\epsilon\kk$ as defined above, and
$$p = p_0 + p_1\ii + p_2\epsilon\jj + p_3\epsilon\kk,$$
where $p_0, p_1, p_2$ and $p_3$ are real numbers.  Then
$$pq = (p_0q_0-p_1q_1) + (p_0q_1+p_1q_0)\ii + (p_0q_2+p_2q_0-p_1q_3+p_3q_1)\epsilon\jj + (p_0q_3+p_3q_0+p_1q_2-p_2q_1)\epsilon\kk,$$
$$qp = (p_0q_0-p_1q_1) + (p_0q_1+p_1q_0)\ii + (p_0q_2+p_2q_0+p_1q_3-p_3q_1)\epsilon\jj + (p_0q_3+p_3q_0-p_1q_2+p_2q_1)\epsilon\kk.$$
In general, $pq \not = qp$.

The conjugate of $q = q_0 + q_1\ii + q_2\epsilon\jj + q_3\epsilon\kk$ is
$$\bar q = q_0 - q_1\ii - q_2\epsilon\jj - q_3\epsilon\kk.$$
The magnitude of $q$ is
$$|q| = \sqrt{q_0^2+q_1^2}.$$  We see that
$$q\bar q = \bar qq = |q|^2,$$
$$|pq|=|qp|=|p|\cdot |q|,$$
and
$$\overline {pq} = \bar q\bar p.$$
A dual complex number $q$ is a real number if and only if $q = \bar q$.
The standard part of $q$ is $q_{st}=q_0 + q_1\ii$.  The infinitesimal part of $q$ is $q_\I = q_2+q_3\ii$.  Note that both $q_{st}$ and $q_\I$ are complex numbers, and
$$q = q_{st} + q_\I\epsilon\jj.$$
We say that a dual complex number $q$ is appreciable if its standard part is nonzero. Otherwise, we say that it is infinitesimal.

It follows the multiplicative inverse of an appreciable dual complex number $q$ is given by
$$q^{-1} = {\bar q \over |q|^2}.$$

Two dual complex numbers $p$ and $q$ are said to be similar if there exists an appreciable dual complex number $u$ such that $u^{-1}pu = q$.   This is denoted as $p \sim q$.  Then $\sim$ is an equivalence relation.  Denote the equivalence class containing $q$ by $[q]$.    If $p \sim q$, then $|p|=|q|$.   Thus, if $p \sim q$, then either both $p$ and $q$ are appreciable, or both of them are infinitesimal.

The primary application of dual complex numbers is in representing rigid body motions in 2D space.  The set of dual complex numbers is a subset of the set of dual quaternions.  For more properties and applications of dual complex numbers, see \cite{Gu11, MKO14}.   In \cite{MKO14}, such dual complex numbers are called anti-commutative dual complex numbers, which parametrize two dimension rotation and translation together.  With this presentation, we can easily interpolate or blend two or more rigid transformations at a low computation cost.   In this paper, we simply call them dual complex numbers.

\subsection{Dual Complex Matrices}

The collections of real, complex and dual complex $m \times n$ matrices are denoted by ${\mathbb R}^{m \times n}$, ${\mathbb C}^{m \times n}$ and ${\mathbb {DC}}^{m \times n}$, respectively.

A dual complex matrix $A= (a_{ij}) \in {\mathbb {DC}}^{m \times n}$ can be denoted as
\begin{equation} \label{e1}
A = A_0 + A_1\ii + A_2\epsilon\jj + A_3\epsilon\kk,
\end{equation}
where $A_0, A_1, A_2, A_3 \in {\mathbb R}^{m \times n}$.   The transpose of $A$ is $A^\top = (a_{ji})$. The conjugate of $A$ is $\bar A = (\bar a_{ij})$.   The conjugate transpose of $A$ is $A^* = (\bar a_{ji}) = \bar A^\top$.

For $A \in {\mathbb {DC}}^{m \times n}$, expressed by (\ref{e1}), its standard part is $A_{st} \equiv {\rm st}(A) = A_0 + A_1\ii$, its infinitesimal part is $A_\I \equiv \I(A) = A_2 + A_3\ii$.  Both $A_{st}$ and $A_\I$ are complex matrices, and
$$A = A_{st} + A_\I\epsilon\jj.$$

%For $A, B \in {\mathbb DC}^{m \times n}$, their inner product is defined as
%$$\langle A, B \rangle = {\rm Tr}(A^*B),$$
%where ${\rm Tr}(A^*B)$ denotes the trace of $A^*B$.
The Frobenius norm of $A$ is
$$\|A\|_F  = \sqrt{\sum_{i=1}^m \sum_{j=1}^n |a_{ij}|^2}.$$
%$$\|A\|_F = \sqrt{\langle A, A \rangle} = \sqrt{{\rm Tr}(A^*A)} = \sqrt{\sum_{i=1}^m \sum_{j=1}^n |a_{ij}|^2}.$$
%the $\ell_1$-norm of $A = (a_{ij}) \in {\mathbb Q}^{m \times n}$ is defined by $\|A\|_1 = \sum_{i=1}^m \sum_{j=1}^n |a_{ij}|$, and the $\ell_\infty$-norm of $A$ is defined by $\|A\|_\infty = \max_{i, j} |a_{ij}|$ \cite{JNS19}.

Let $A \in {\mathbb {DC}}^{m \times n}$ and $B \in {\mathbb {DC}}^{n \times r}$.   Then we have $(AB)^* = B^*A^*$.   But in general, $(AB)^\top \not = B^\top A^\top$ and $\overline {AB} \not = \bar A \bar B$ in general.

A square dual complex matrix $A \in {\mathbb {DC}}^{n \times n}$ is called Hermitian if $A^* = A$; unitary if $A^*A = I$; and invertible (nonsingular) if $AB = BA = I$ for some $B \in {\mathbb {DC}}^{n \times n}$.
We have $(AB)^{-1} = B^{-1}A^{-1}$ if $A$ and $B$ are invertible, and $\left(A^*\right)^{-1} = \left(A^{-1}\right)^*$ if $A$ is invertible.

Applications of dual complex matrices include classical mechanics and robotics, complex representations of the Lorentz group in relativity and electrodynamics, conformal mappings in computer vision, the physics of scattering processes, etc., see \cite{Br20}.

\subsection{Dual Complex Vectors}

Denote $\vx \in {\mathbb {DC}}^{n \times 1}$ and $\vx^\top \in {\mathbb {DC}}^{1 \times n}$ for column and row dual complex vectors.   We say that $\vx \in {\mathbb {DC}}^{n \times 1}$ is appreciable if at least one of its component is appreciable.   The magnitude of $\vx \in {\mathbb {DC}}^{n \times 1}$ is
$$\|\vx \| = \sqrt{\vx^*\vx}.$$
If $\|\vx\| = 1$, then we say that $\vx$ is a unit column vector.  If $\vx, \vy \in {\mathbb {DC}}^{n \times 1}$ and $\vx^*\vy = 0$, then we say that $\vx$ and $\vy$ are orthogonal to each other.  If
$\vx^{(1)}, \cdots, \vx^{(n)} \in {\mathbb {DC}}^{n \times 1}$ and $\left(\vx^{(i)}\right)^*\vx^{(j)} = \delta_{ij}$ for $i, j = 1, \cdots, n$, where $\delta_{ij}$ is the Kronecker symbol, then we say that
$\{ \vx^{(1)}, \cdots, \vx^{(n)} \}$ is an orthonormal basis of ${\mathbb {DC}}^{n \times 1}$.
A square dual complex matrix $A \in {\mathbb {DC}}^{n \times n}$ is unitary if and only if its column vectors form an orthonormal basis of ${\mathbb {DC}}^{n \times 1}$.

The following properties for complex matrices still hold for dual complex matrices:

1. If $A \in {\mathbb {DC}}^{n \times n}$ and $\vx \in {\mathbb {DC}}^{n \times 1}$, then
$$\|A\vx\| \le \|A\|_F \|\vx\|.$$

2. If $U \in {\mathbb {DC}}^{n \times n}$ is unitary and $\vx \in {\mathbb {DC}}^{n \times 1}$, then
$$\|U\vx\| = \|\vx\|.$$

3. (Unitary Invariance) If $U \in {\mathbb {DC}}^{m \times m}$ and $V \in {\mathbb {DC}}^{n \times n}$ are unitary, and $A \in {\mathbb {DC}}^{m \times n}$, then
$$\|UAV\|_F = \|A\|_F.$$
The proof of the unitary invariance property is the same as the proof of orthogonal invariance in \cite{GV13}.

Note that if both $\vx$ and $q$ are appreciable, then $\vx q$ is appreciable.

The following proposition can be proved by definition directly.

\begin{Prop} \label{p2.0}
Suppose that $A \in {\mathbb {DC}}^{n \times n}$ is invertible.   Then all column and row vectors of $A$ are appreciable.
\end{Prop}

\section{Right Eigenvalues of Dual Complex Matrices}

Suppose that $A \in {\mathbb {DC}}^{n \times n}$.   If there are $\lambda \in \mathbb {DC}$ and $\vx \in
{\mathbb {DC}}^{n \times 1}$, where $\vx$ is appreciable, such that
\begin{equation} \label{e2}
A\vx = \vx\lambda,
\end{equation}
then we say that $\lambda$ is a right eigenvalue of $A$, with $\vx$ as a corresponding right eigenvector.

Here, we request that $\vx$ is appreciable.   See the definition of eigenvectors of dual number matrices in \cite{Gu21}.

Note that $A\vx = \vx\lambda$ implies $A(\vx q) = (A\vx)q = \vx \lambda q = (\vx q)(q^{-1}\lambda q)$ if $q$ is appreciable.   Thus, if $\lambda$ is a right eigenvalue of $A$, then any member of $[\lambda]$ is a right eigenvalue of $A$.

We may also define left eigenvalues, but we do not go to this direction, as it is not related with our discussion.

We prove a lemma.

\begin{Lem} \label{l3.1}
Let $\va \in \mathbb {C}^n$.  Then $\jj \va = \bar \va \jj$.
\end{Lem}
\begin{proof}
Let $\va = \vb + \vc\ii$, where $\vb$ and $\vc$ are real vectors.  Then
$$\jj \va = \vb\jj - \vc\kk = (\vb - \vc\ii)\jj = \bar \va \jj.$$
\end{proof}

We now establish conditions for a dual complex number to be a right eigenvalue of a square dual complex matrix.

We have the following theorem.

\begin{Thm} \label{t3.2}
Suppose that $A = A_{st}+A_\I\epsilon\jj \in {\mathbb {DC}}^{n \times n}$.   Then $\lambda = \lambda_{st} + \lambda_\I \epsilon \jj$ is a right eigenvalue of $A$ with a right eigenvector $\vx = \vx_{st}+\vx_I\epsilon\jj$ only if $\lambda_{st}$ is an eigenvalue of the complex matrix $A_{st}$ with an
eigenvector $\vx_{st}$, i.e., $\vx_{st} \not = \0$ and
\begin{equation} \label{e3}
A_{st}\vx_{st} = \lambda_{st} \vx_{st}.
\end{equation}
Furthermore, if $\lambda_{st}$ is an eigenvalue of the complex matrix $A_{st}$ with an
eigenvector $\vx_{st}$, then $\lambda$ is a right eigenvalue of $A$ with a right eigenvector $\vx$ if and only if $\lambda_\I$ and $\vx_\I$ satisfy
\begin{equation} \label{e4}
\lambda_\I\vx_{st} = A_\I\bar \vx_{st} + A_{st}\vx_\I - \bar \lambda_{st} \vx_\I.
\end{equation}
\end{Thm}
\begin{proof}
By definition, $\lambda$ is a right eigenvalue of $A$ with a right eigenvector $\vx$ if and only if $\vx_{st} \not = \0$ and $A\vx = \vx \lambda$.   Then $A\vx = \vx \lambda$ is equivalent to
$$(A_{st}+A_\I\epsilon\jj)(\vx_{st}+\vx_I\epsilon\jj) = (\vx_{st}+\vx_I\epsilon\jj)(\lambda_{st} + \lambda_\I \epsilon \jj).$$
This is further equivalent to $A_{st}\vx_{st} = \vx_{st}\lambda_{st}$, i.e., (\ref{e3}), and
\begin{equation} \label{e5}
A_{st}\vx_\I\epsilon \jj + A_\I \epsilon \jj \vx_{st} = \vx_{st} \lambda_\I \epsilon \jj + \vx_\I \epsilon\jj \lambda_{st}.
\end{equation}
By Lemma \ref{l3.1}, (\ref{e5}) is equivalent to
$$A_{st}\vx_\I \jj + A_\I \bar \vx_{st}\jj = \vx_{st} \lambda_\I \jj + \vx_\I \bar \lambda_{st}\jj,$$
which is further equivalent to (\ref{e4}).    The conclusions of this theorem follow from
these.
\end{proof}

Let $\lambda \equiv \lambda_{st}$ and $\lambda_\I = 0$ in this theorem.   We have the following corollary.

\begin{Cor} \label{c3.3}
Suppose that $A = A_{st}+A_\I\epsilon\jj \in {\mathbb {DC}}^{n \times n}$.   Then  a complex number  $\lambda$ is a right eigenvalue of $A$ with a right eigenvector $\vx_{st}+\vx_I\epsilon\jj$ only if $\lambda$ is an eigenvalue of the complex matrix $A_{st}$ with an
eigenvector $\vx_{st}$, i.e., $\vx_{st} \not = \0$ and
\begin{equation} \label{e6}
A_{st}\vx_{st} = \lambda \vx_{st}.
\end{equation}
This indicates that $A$ has at most $n$ complex right eigenvalues.   Furthermore, if $\lambda$ is an eigenvalue of the complex matrix $A_{st}$ with an
eigenvector $\vx_{st}$, then $\lambda$ is a right eigenvalue of $A$ with a right eigenvector $\vx$ if and only if $\vx_\I$ satisfies
\begin{equation} \label{e7}
A_\I\bar \vx_{st} + A_{st}\vx_\I - \bar \lambda \vx_\I = \0.
\end{equation}
\end{Cor}

A square quaternion matrix always has complex right eigenvalues \cite{Br51, WLZZ18, Zh97}.   A square dual complex matrix may have no complex right eigenvalue at all.  This is very different.   See the following example.

{\bf Example 1}  Let $A = I_n+I_n\epsilon \jj$, i.e., $A_{st} = A_\I = I_n$.  Then the only eigenvalue of $A_{st}$ is $\lambda_{st} = 1$ with multiplicity $n$.   Then (\ref{e7}) is equivalent to $\bar \vx_{st} = \0$, which contradicts that $\vx_{st} \not = \0$.  Thus, $A$ has no complex right eigenvalue.   However, let $\vx = \ve + \ve \ii$, where $\ve = (1, \cdots, 1)^\top$,  and
$$\lambda = 1 -\ii\epsilon \jj.$$
 Then we see that $\lambda$ is a right eigenvalue of $A$ with $\vx$ as its right eigenvector.

In the next section, we will give an example that a dual complex Hermitian matrix has no right eigenvalue at all.

\section{Right Eigenvalues and Right Eigenvectors of Hermitian Matrices}

Suppose that $A \in {\mathbb {DC}}^{n \times n}$ is a Hermitian matrix.   Then for $\vx \in {\mathbb {DC}}^{n \times 1}$, $\vx^* A \vx = (\vx^* A \vx)^*$ is a real number.  We say that $A$ is positive semi-definite if for all $\vx \in {\mathbb {DC}}^{n \times 1}$, $\vx^* A \vx \ge 0$.    We say that $A$ is positive definite if for all $\vx \in {\mathbb {DC}}^{n \times 1}$ and $\vx$ is appreciable, $\vx^* A \vx > 0$.

By definition, we have the following proposition.

\begin{Prop} \label{p4.1}
A dual complex matrix $A = A_{st} + A_\I\epsilon\jj \in {\mathbb {DC}}^{n \times n}$ is a Hermitian matrix if and only if $A_{st}$ is a complex Hermitian matrix and $A_\I$ is a skew-symmetric complex matrix, i.e., $A_{st}^* = A_{st}$ and $A_\I^\top = - A_\I$.
\end{Prop}

Then we have the following proposition.

\begin{Prop} \label{p4.2}
A right eigenvalue $\lambda$ of a Hermitian matrix $A = A_{st} + A_\I\epsilon\jj \in {\mathbb {DC}}^{n \times n}$ must be a real number and is an eigenvalue of the complex Hermitian matrix $A_{st}$.   Thus,
a dual complex Hermitian matrix has at most $n$ real right eigenvalues and no other right eigenvalues.

A right eigenvalue of a positive semi-definite Hermitian matrix $A \in {\mathbb {DC}}^{n \times n}$ must be a nonnegative number.   Thus, a dual complex positive semi-definite Hermitian matrix has at most $n$ nonnegative right eigenvalues and no other right eigenvalues.

A right eigenvalue of a positive definite Hermitian matrix $A \in {\mathbb {DC}}^{n \times n}$ must be a positive number.  Thus, a dual complex positive definite Hermitian matrix has at most $n$ positive right eigenvalues and no other right eigenvalues.
\end{Prop}
\begin{proof}
Suppose that $A \in {\mathbb {DC}}^{n \times n}$ is a Hermitian matrix, $\lambda$ is a right eigenvalue of $A$ with right eigenvectors $\vx$.   Then
$$\vx^*A\vx = \vx^*\vx\lambda = \|\vx\|^2\lambda$$
is a real number.  Thus, $\lambda$ is a real number.  By Corollary \ref{c3.3} and Proposition \ref{p4.1}, $\lambda$ is an eigenvalue of the complex Hermitian matrix $A_{st}$.   Thus, a dual complex Hermitian matrix has at most $n$ real right eigenvalues and no other right eigenvalues.

The other conclusions can be proved similarly.
\end{proof}

We now give an example that a dual complex Hermitian matrix has no right eigenvalue at all.

{\bf Example 2}  Let $n=2$ and
$$A = \begin{bmatrix} 1 & \epsilon\jj \\ -\epsilon \jj & 1 \end{bmatrix}.$$
By definition, $A$ is a Hermitian matrix.  Denote
$$A_{st} = I_2 = \begin{bmatrix} 1 & 0 \\ 0 & 1 \end{bmatrix} \ {\rm and}\ A_\I = \begin{bmatrix} 0 & 1 \\ -1 & 0 \end{bmatrix}.$$
Suppose that $A$ has a right eigenvalue $\lambda$.  By Proposition \ref{p4.2}, $\lambda$ is a real number and is an eigenvalue of $I_2$.  Thus, $\lambda = 1$.  Assume it has a right eigenvector $\vx = \vx_{st} + \vx_\I\epsilon \jj$.  By Corollary \ref{c3.3}, it needs to satisfy (\ref{e7}), i.e., $A_\I \bar \vx_{st} = \0$.   This implies that $\vx_{st} = \0$.   Then $\vx$ is not appreciable, and is not a right eigenvector.
Therefore, $A$ has no right eigenvalue at all.

Fortunately, the above example only occurs in the case that $A_{st}$ has multiple eigenvalues.   Recall that an eigenvalue with multiplicity $1$ is called a simple eigenvalue.

\begin{Thm} \label{t4.3}
Let $A = A_{st} + A_\I\epsilon\jj \in {\mathbb {DC}}^{n \times n}$ be a Hermitian matrix.  If $\lambda$ is a simple eigenvalue of $A_{st}$, then $\lambda$ is a right eigenvalue of $A$.
\end{Thm}
\begin{proof}  By Proposition \ref{p4.1}, $A_{st}$ is a complex Hermitian matrix and $\lambda$ is real.
Then (\ref{e7}) becomes
\begin{equation} \label{e8}
A_\I\bar \vx_{st} = (\lambda I_n - A_{st})\vx_\I.
\end{equation}

Since $\lambda$ is a simple eigenvalue of $A_{st}$, (\ref{e8}) has a solution $\vx_\I$ if and only if
$$\vx_{st}^*A_\I\bar \vx_{st} = 0.$$
But
$$\vx_{st}^* A_\I \bar \vx_{st} = \left(\vx_{st}^* A_\I \bar \vx_{st}\right)^\top = \vx_{st}^* (A_\I)^\top \bar \vx_{st} = - \vx_{st}^* A_\I \bar \vx_{st},$$
as by Proposition \ref{p4.1}, $A_\I^\top = -A_\I$.
Thus,
$$\vx_{st}^*A_\I\bar \vx_{st} = 0.$$
This implies (\ref{e8}) has a solution $\vx_\I$.   By Corollary \ref{c3.3}, $\lambda$ is a right eigenvalue of $A$.
\end{proof}

We now consider right eigenvectors of a dual complex Hermitian matrix.

\begin{Prop} \label{p4.0}
Suppose that $A \in {\mathbb {DC}}^{n \times n}$ is a Hermitian matrix, and has a right eigenvalue $\lambda$, with right eigenvectors
$\vx^{(1)}, \cdots, \vx^{(k)}$.   Then $\vy = \sum_{j=1}^k \vx^{(j)}\alpha_j$ is a right eigenvector of $A$, associated with $\lambda$, as long as $\vy$ is appreciable, where $\alpha_1, \cdots, \alpha_k \in \mathbb {DC}$.
\end{Prop}
\begin{proof}
We have
$$A\vy = A \sum_{j=1}^k \vx^{(j)}\alpha_j = \sum_{j=1}^k A \vx^{(j)}\alpha_j = \sum_{j=1}^k \vx^{(j)}\lambda \alpha_j = \sum_{j=1}^k \vx^{(j)} \alpha_j\lambda = \vy\lambda.$$
Hence, $\vy$ is a right eigenvector of $A$, associated with $\lambda$, as long as $\vy$ is appreciable.   Note that by Proposition \ref{p4.2}, $\lambda$ is a real number.  Thus, we have
$\lambda \alpha_j = \alpha_j\lambda$.   %This is not true for general dual complex matrices.
\end{proof}

Suppose that $A \in {\mathbb {DC}}^{n \times n}$ is a Hermitian matrix and has a real right eigenvalue $\lambda$.   Denote the set of right eigenvectors of $A$, associated with $\lambda$ by $V_A(\lambda)$.

The right eigenvectors of a dual complex Hermitian matrix have properties similar to eigenvectors of a complex Hermitian matrix.

\begin{Prop} \label{p4.5}
Two right eigenvectors of a Hermitian matrix $A \in {\mathbb {DC}}^{n \times n}$, associated with two distinct right eigenvalues, are orthogonal to each other.
\end{Prop}
\begin{proof}
Suppose $\vx$ and $\vy$ are two right eigenvectors of a Hermitian matrix $A \in {\mathbb {DC}}^{n \times n}$, associated with two distinct right eigenvalues $\lambda$ and $\mu$, respectively.   By Proposition \ref{p4.2}, $\lambda$ and $\mu$ are real numbers.   Then $\lambda \not = \mu$.
We have
$$\lambda(\vx^*\vy) = (\vx\lambda)^*\vy = (A\vx)^*\vy = \vx^*A\vy = \vx^*\vy\mu = \mu \vx^*\vy.$$
Since $\lambda \not = \mu$, we have $\vx^*\vy = 0$.
\end{proof}

Suppose that  $A \in {\mathbb {DC}}^{n \times n}$ is a Hermitian matrix, and $\lambda$ is an eigenvalue of
$A_{st}$ with multiplicity $p$.   Then $\lambda$ is a real number.   If $\lambda$ is also a right eigenvalue of $A$, and there are right eigenvectors $\{ \vu^{(1)}, \cdots, \vu^{(p)} \}$ of $A$ associated
with $\lambda$ such that $\{ \vu^{(1)}, \cdots, \vu^{(p)} \}$ forms an orthonormal basis of $V_A(\lambda)$,
% of the eigenspace of $\lambda$ with respect to $A_{st}$,
then we say that
$\lambda$ is a {\bf regular} eigenvalue of $A_{st}$.    By Theorem \ref{t4.3}, any single multiple eigenvalue of $A_{st}$ is regular.   If all the eigenvalues of $A_{st}$ are regular, then we say that the Hermitian matrix $A$ is {\bf regular}, i.e., $A$ has $n$ right eigenvalues, counting with multiplicity.
Otherwise, we say that $A$ is irregular.

\begin{Thm} \label{t4.6}
Suppose that $A \in {\mathbb {DC}}^{n \times n}$ is a regular Hermitian matrix.  Then there are unitary matrix $U \in {\mathbb {DC}}^{n \times n}$ and real diagonal matrix $D \in {\mathbb {R}}^{n \times n}$ such that $D = U^{-1}AU$.  The $n$ diagonal entries of $D$ are right eigenvalues of $A$, and the column vectors of $U$ are corresponding right eigenvectors.
\end{Thm}
\begin{proof}
By above discussion, ${\mathbb {DC}}^{n \times 1}$ has an orthonormal basis such that the basis vectors are right eigenvectors of $A$.   Let them be the column vectors of $U$.  Then we have the desired result.
\end{proof}

\section{Right Subeigenvalues and Subeigenvectors of Hermitian Matrices}

Suppose that $A \in {\mathbb {DC}}^{n \times n}$ is a Hermitian matrix.   If $A$ is irregular, then $A$ has no $n$ right eigenvalues, and $A$ cannot be diagonalized.  We also know that the right eigenvalues of $A$ must be eigenvalues of $A_{st}$, but not vice versa.   Where are those ``missing'' eigenvalues of $A_{st}$?   In this section, we will find out those ``missing'' eigenvalues of $A_{st}$ and their roles in the spectral theory of dual complex Hermitian matrices.   We call them right subeigenvalues of $A$.

Suppose that $A \in {\mathbb {DC}}^{n \times n}$.  Assume that there is $\lambda \in {\mathbb {DC}}$,
$\mu \in {\mathbb {C}}$, $\mu \not = 0$, and $\vx, \vy \in {\mathbb {DC}}^{n \times 1}$ such that
\begin{equation} \label{e9}
A\vx = \vx\lambda + \vy\mu \epsilon \jj,
\end{equation}
and
\begin{equation} \label{e10}
A\vy = \vy\lambda - \vx\mu \epsilon \jj,
\end{equation}
where $\vx$ and $\vy$ are appreciable, and orthogonal to each other.
Then we say that $\lambda$ is a right subeigenvalue of $A$, with multiplicity $2$, $\vx$ and $\vy$ are right subeigenvectors of $A$, associated with $\lambda$, and $\mu$ is the adjoint parameter of $A$, associated with $\lambda$.

Here, we require $\mu$ to be a complex number, as if $\mu$ is a dual complex number, its infinitesimal part does not play any role here.     We have the following proposition parallel to Proposition \ref{p4.2}.

\begin{Prop} \label{p5.1}
A right subeigenvalue $\lambda$ of a Hermitian matrix $A = A_{st} + A_\I\epsilon\jj \in {\mathbb {DC}}^{n \times n}$ must be a real number and is a multiple eigenvalue of the complex Hermitian matrix $A_{st}$.

A right subeigenvalue of a positive semi-definite Hermitian matrix $A \in {\mathbb {DC}}^{n \times n}$ must be a nonnegative number.   A right subeigenvalue of a positive definite Hermitian matrix $A \in {\mathbb {DC}}^{n \times n}$ must be a positive number.
\end{Prop}
\begin{proof}
Suppose that $A \in {\mathbb {DC}}^{n \times n}$ is a Hermitian matrix, $\lambda$ is a right subeigenvalue of $A$ with right eigenvectors $\vx$ and $\vy$, and the adjoint parameter $\mu$.   By (\ref{e9}) and $\vx^*\vy = 0$, we have
$$\vx^*A\vx = \vx^*\vx\lambda = \|\vx\|^2\lambda$$
is a real number.  Thus, $\lambda$ is a real number.  With $A = A_{st} + A_\I\epsilon\jj$, $\vx = \vx_{st} + \vx_\I\epsilon\jj$ and  $\vy = \vy_{st} + \vy_\I\epsilon\jj$, from (\ref{e9}) and (\ref{e10}), we have
$$A_{st}\vx_{st} = \vx_{st}\lambda = \lambda\vx_{st}, \ \  A_{st}\vy_{st} = \vy_{st}\lambda = \lambda\vy_{st}.$$
Since $\vx$ and $\vy$ are appreciable, $\vx_{st} \not = \0$ and $\vy_{st} \not = \0$.   By $\vx^*\vy = 0$, we have $\vx^*_{st}\vy_{st} = 0$.   Thus, $\vx_{st}$ and $\vy_{st}$ are two orthogonal eigenvectors of $A_{st}$, associated with $\lambda$.  Then $\lambda$ is a multiple eigenvalue of $A_{st}$.

The other conclusions can be proved similarly.
\end{proof}

We will see that ``missing'' eigenvalues of $A_{st}$ are recovered as right subeigenvalues of $A$, if they are not right eigenvalues of $A$.   Consider the case that $\lambda$ is a double eigenvalue of $A_{st}$.   This is the case of Example 2.

\begin{Prop} \label{p5.2}
Suppose that $A = A_{st} + A_\I\epsilon\jj \in {\mathbb {DC}}^{n \times n}$ is a Hermitian matrix, and $\lambda$ is a double eigenvalue of $A_{st}$.   Let $\{ \vx_{st}, \vy_{st} \}$ be an orthonormal basis of the eigenspace of $\lambda$ with respect to $A_{st}$, and
\begin{equation} \label{e11}
\mu = \vy_{st}^*A_\I\bar \vx_{st}.
\end{equation}
Then $\lambda$ is a double right eigenvalue of $A$ if and only if $\mu = 0$.   Otherwise, $\lambda$ is a double right subeigenvalue of $A$, with $\mu$ as its adjoint parameter.  This fact is independent of the choice of the orthonormal basis $\{ \vx_{st}, \vy_{st} \}$.
\end{Prop}
\begin{proof}  Since $A_{st}\vx_{st} = \lambda\vx_{st}$, by some algebraic derivation, we see that (\ref{e9}) is equivalent to
\begin{equation} \label{e12}
(\lambda I_n - A_{st})\vx_\I = A_\I \bar \vx_{st} - \mu \vy_{st}.
\end{equation}
Similarly, (\ref{e10}) is equivalent to
\begin{equation} \label{e13}
(\lambda I_n - A_{st})\vy_\I = A_\I \bar \vy_{st} + \mu \vx_{st}.
\end{equation}
By matrix theory, (\ref{e12}) and (\ref{e13}) have solutions $\vx_\I$ and $\vy_\I$ if and only if
$$\vx_{st}^*(A_\I \bar \vx_{st} - \mu \vy_{st}) = 0,\ \ \vy_{st}^*(A_\I \bar \vx_{st} - \mu \vy_{st}) = 0,$$
and
$$\vx_{st}^*(A_\I \bar \vy_{st} + \mu \vx_{st}) = 0,\ \ \vy_{st}^*(A_\I \bar \vy_{st} + \mu \vx_{st}) = 0.$$
We may prove these by using (\ref{e11}), the orthogonality between $\vx_{st}$ and $\vy_{st}$, and the skew-symmetry of $A_\I$.  Then, by (\ref{e9}) and (\ref{e10}), $\lambda$ is a double right eigenvalue or subeigenvalue of $A$, depending upon $\mu=0$ or not.   Finally, we may prove the fact that $\mu = 0$ or not is independent from the choice of the orthonormal basis $\{ \vx_{st}, \vy_{st} \}$, by choosing another orthonormal basis to show this, via the skew-symmetry of $A_\I$.
\end{proof}

When the multiplicity of an eigenvalue of $A_{st}$ is higher than $2$, the situation is more complicated. Viewing Theorem \ref{t4.6}, Proposition \ref{p5.1}, and Theorem 3.1 of \cite{Gu21}, we have the following theorem.   In the proof of this theorem, the symbols $I_i$ and $I$ are different from the rest of this paper.

\begin{Thm} \label{t5.2}
Suppose that $A \in {\mathbb {DC}}^{n \times n}$ is a Hermitian matrix.  Then there are unitary matrix $U \in {\mathbb {DC}}^{n \times n}$ and block-diagonal matrix $\Sigma \in {\mathbb {DC}}^{n \times n}$ such that $\Sigma = U^{-1}AU$, and each block of $\Sigma$ is either of the form:
\begin{itemize}

\item $(\lambda_i)$,

\item $\begin{bmatrix} \lambda_i & \mu_i\epsilon \jj\\ -\mu_i\epsilon \jj& \lambda_i \end{bmatrix}$,
\end{itemize}
where each $\lambda_i$ is real, $\mu_i$ is complex, and $\mu_i \not = 0$.
\end{Thm}
\begin{proof}   Let $A \in {\mathbb {DC}}^{n \times n}$ be a Hermitian matrix.  Denote $A = A_{st} + A_\I \epsilon \jj$, where $A_{st}, A_\I \in {\mathbb {C}}^{n \times n}$.  Then $A_{st}$ is Hermitain, and $A_\I$ is skew-symmetric.   This implies that there is a complex unitary matrix $S \in {\mathbb {C}}^{n \times n}$ and a real diagonal matrix $D \in {\mathbb {R}}^{n \times n}$ such that $D = SA_{st}S^*$.
Suppose that $D = {\rm diag}(\lambda_1I_1, \lambda_2I_2, \cdots, \lambda_rI_r)$, where $\lambda_1 > \lambda_2 > \cdots > \lambda_r$, and $I_1, \cdots, I_r$ are identity matrices such that the sum of their dimensions is $n$.  Let $M = SAS^*$.  Then
\begin{eqnarray*}
&& M \\ & = & D + SA_\I \epsilon \jj S^* = D + SA_\I S^\top \epsilon \jj\\
& = & \begin{bmatrix}
\lambda_1I_1 + \epsilon C_{11}\jj & \epsilon C_{12}\jj & \cdots & \epsilon C_{1r}\jj \\
-\epsilon C_{12}^\top\jj & \lambda_2I_2 + \epsilon C_{22}\jj & \cdots  & \epsilon C_{2r}\jj \\
\vdots & \vdots & \ddots & \vdots \\
-\epsilon C_{1r}^\top\jj & -\epsilon C_{2r}^\top\jj &  \cdots & \lambda_r I_r + \epsilon C_{rr}\jj
\end{bmatrix},
\end{eqnarray*}
where each $C_{ij}$ is a complex matrix of adequate dimensions, each $C_{ii}$ is skew-symmetric.

Let
$$P = \begin{bmatrix}
I_1 & {\epsilon C_{12}\jj \over \lambda_1-\lambda_2} & \cdots & {\epsilon C_{1r}\jj \over \lambda_1-\lambda_r}\\
{\epsilon C_{12}^\top\jj \over \lambda_1-\lambda_2} & I_2 & \cdots  & {\epsilon C_{2r}\jj \over \lambda_2-\lambda_r} \\
\vdots & \vdots & \ddots & \vdots \\
{\epsilon C_{1r}^\top\jj \over \lambda_1-\lambda_r} & {\epsilon C_{2r}^\top\jj \over \lambda_2-\lambda_r}  &  \cdots & I_r
\end{bmatrix},$$
and $N=PMP^* \equiv (PS)A(PS)^*$.   Then
$$P^* = \begin{bmatrix}
I_1 & -{\epsilon C_{12}\jj \over \lambda_1-\lambda_2} & \cdots & -{\epsilon C_{1r}\jj \over \lambda_1-\lambda_r}\\
-{\epsilon C_{12}^\top\jj \over \lambda_1-\lambda_2} & I_2 & \cdots  & -{\epsilon C_{2r}\jj \over \lambda_2-\lambda_r} \\
\vdots & \vdots & \ddots & \vdots \\
-{\epsilon C_{1r}^\top\jj \over \lambda_1-\lambda_r} & -{\epsilon C_{2r}^\top\jj \over \lambda_2-\lambda_r}  &  \cdots & I_r
\end{bmatrix},$$
$PP^* = P^*P = I$, where $I$ is the $n \times n$ identity matrix.   Thus, $P$ and $PS$ are unitary matrices.  Also, $N$ is a block diagonal matrix, $N= {\rm diag}(\lambda_1I_1+\epsilon C_{11}\jj, \lambda_2I_2 +\epsilon C_{22}\jj, \cdots, \lambda_rI_r+\epsilon C_{rr}\jj)$.   We then may use the spectral theorem for skew-symmetric matrices \cite{Hu44, Yo61} to diagonalize $N$ further to the desired form.
\end{proof}

With Theorem \ref{t4.6}, Proposition \ref{p5.1} and Theorem \ref{t5.2}, we have the following theorem.

\begin{Thm} \label{t5.3}
Suppose that $A \in {\mathbb {DC}}^{n \times n}$ is Hermitian.  Then $A$ has exactly $n$ right eigenvalues and subeigenvalues, which are all real numbers.  There are also $n$ right eigenvectors and subeigenvectors, associated with these $n$ right eigenvalues and subeigenvalues, such that they form an orthonormal basis of ${\mathbb {DC}}^{n \times 1}$.   The Hermitian matrix $A$ is positive semi-definite or definite if and only all of these right eigenvalues and subeigenvalues are nonnegative or positive, respectively.
\end{Thm}
\begin{proof}  In the decomposition $\Sigma = U^{-1}AU$ in Theorem \ref{t5.2}, we see that $\lambda_i$ in the block $(\lambda_i)$ is a right eigenvalue of $A$, while $\lambda_i$ in the block
$$\begin{bmatrix} \lambda_i & \mu_i\epsilon \jj\\ -\mu_i\epsilon \jj& \lambda_i \end{bmatrix}$$
is a right subeigenvalue of $A$, with multiplicity $2$, and the columns of $U$ are corresponding right eigenvectors and subeigenvectors.    The other conclusions follow from Propositions \ref{p4.2} and \ref{p5.1}.
\end{proof}

The following corollary complements Theorem \ref{t4.3} and Proposition \ref{p5.2}.

\begin{Cor} \label{c5.5}
Suppose that $A  = A_{st} + A_\I\epsilon\jj \in {\mathbb {DC}}^{n \times n}$ is Hermitian, and $\lambda$ is an eigenvalue of $A_{st}$ with multiplicity $p$.   Then there is an integer $k$ satisfying $0 \le 2k \le p$, such that $\lambda$ is a right eigenvalue of $A$, with multiplicity $p-2k$, if $p-2k > 0$, and
a right subeigenvalue of $A$, with multiplicity $2k$, if $k > 0$.
\end{Cor}

\section{Singular Value Decomposition of Dual Complex Matrices}

Suppose that $A \in {\mathbb {DC}}^{m \times n}$.   Then $A^*A \in {\mathbb {DC}}^{n \times n}$ is a positive semi-definite Hermitian matrix.   By Theorem \ref{t5.3}, $A^*A$ has exactly $n$ right eigenvalues and subeigenvalues, and they are all nonnegative numbers.    Assume that $A^*A$ has $r$ positive right eigenvalues and subeigenvalues.   Then we say that the standard rank of $A$ is $r$.

We now present the singular value decomposition of dual complex matrices.   The following theorem generalizes Theorem 4.4 of \cite{Gu21} for square dual number matrices to general dual complex matrices.

\begin{Thm}
Suppose that $A \in {\mathbb {DC}}^{m \times n}$ with standard rank $r$.   Then there exist a dual complex unitary matrix $U \in {\mathbb {DC}}^{m \times m}$ and a dual complex unitary matrix $V \in {\mathbb {DC}}^{n \times n}$ such that
\begin{equation} \label{e14}
U^*AV = \begin{bmatrix} \Sigma_r & O & O \\ O & D\epsilon\jj & O \\ O & O & O \end{bmatrix},
\end{equation}
where $\Sigma_r \in {\mathbb {DC}}^{r \times r}$ is a block-diagonal matrix, and each block of $\Sigma_r$ is either of the form:
\begin{itemize}

\item $(\sigma_i)$,

\item $\begin{bmatrix} \sigma_i & \nu_i\epsilon \jj\\ -\nu_i\epsilon \jj& \sigma_i \end{bmatrix}$,

\end{itemize}
each $\sigma_i$ is real and positive, $\nu_i$ is complex, $\nu_i \not = 0$, and $D$ is a $p \times p$ positive diagonal matrix, $r+p \le l = \min \{ m, n \}$.
\end{Thm}
\begin{proof}   Suppose that positive semi-definite Hermitian matrix $A^*A$ has right eigenvalues and subeigenvalues $\sigma_1^2, \cdots, \sigma_n^2$, satisfying $\sigma_1 \ge \cdots \ge \sigma_r > 0$, $\sigma_{r+1} = \cdots = \sigma_n = 0$.   By Theorem \ref{t5.3}, there are orthogonal right eigenvectors
and subeigenvectors $\vv^{(1)}, \cdots, \vv^{(n)}$ of $A^*A$, associated with $\sigma_1^2, \cdots, \sigma_n^2$, respectively.   Write
$V'_1 = (\vv^{(1)}, \cdots, \vv^{(r)})$,  $V'_2 = (\vv^{(r+1)}, \cdots, \vv^{(n)})$, $V' = (V'_1, V'_2)$. Let $\Sigma_r$ be block-diagonal matrix described in the theorem. Then we have
$$A^*AV'_1 = V'_1\Sigma_r^2,$$
$$(V'_1)^*A^*AV'_1 = \Sigma_r^2,$$
$$A^*AV'_2 = O,$$
$$(V'_2)^*A^*AV'_2 = O.$$
Therefore, $AV'_2 = B\epsilon\jj$ for a complex matrix $B$.  Let $U'_1 = AV'_1\Sigma_r^{-1}$.
Then
$$(U'_1)^*AV'_2 = \left(\Sigma_r^{-1}\right)^*(V'_1)^*(A^*AV'_2) = \left(\Sigma_r^{-1}\right)^*(V'_1)^* O =  O.$$
Also observe that $(U'_1)^*U'_1 = I_r$.  Take $U'_2 \in {\mathbb {DC}}^{m \times (m-r)}$ such that $U' = (U'_1, U'_2)$ is a unitary matrix.  We see that
$$(U'_2)^*B\epsilon\jj = G\epsilon \jj,$$
where $G$ is an $(m-r) \times (n-r)$ complex matrix.   Then
\begin{eqnarray*}
(U')^*AV' & = & \begin{bmatrix} (U'_1)^*AV'_1 & (U'_1)^*AV'_2 \\ (U'_2)^*AV'_1 & (U'_2)^*AV'_2 \end{bmatrix}\\
& = & \begin{bmatrix} (U'_1)^*U'_1\Sigma_r & O \\ (U'_2)^*U'_1\Sigma_r &  G\epsilon\jj \end{bmatrix}\\
& = & \begin{bmatrix} \Sigma_r & O \\ O & G\epsilon\jj \end{bmatrix}.
\end{eqnarray*}
After taking SVD decomposition of the complex matrix $G$, we have the desired result.
\end{proof}

We call the decomposition (\ref{e14}) the singular value decomposition (SVD) of the matrix $A$.   Assume that $D = {\rm diag}(\sigma_{r+1}, \cdots, \sigma_{r+p})$, where  $\sigma_{r+1} \ge \cdots \ge \sigma_{r+p} > 0$, and $r+p \le l = \min \{ m, n \}$.  We call $\sigma_1, \cdots, \sigma_r$ the standard singular values of $A$, $\sigma_{r+1}, \cdots, \sigma_{r+p}$ the infinitesimal singular values of $A$, and $\sigma_{r+p+1}=  \cdots = \sigma_l = 0$
 the zero singular values of $A$, and $p$ the infinitesimal rank of $A$.

%\section{Dual Quaternion Matrices}

%In this section, we discuss possible extension of our work to dual quaternion matrices.

\section{Extensions to Dual Quaternion Matrices}

In this section, we briefly study extensions of our results to dual quaternion matrices.

We denote the quaternions and the dual quaternions by ${\mathbb Q}$ and $\mathbb {DQ}$, respectively.
A dual quaternion $q$ has the form
$$q = q_{st} + q_\I\epsilon,$$
where $q_{st}, q_\I \in {\mathbb Q}$ are the standard part and the infinitesimal part of $q$ respectively,
$$q_{st} = q_0+q_1\ii+q_2\jj+q_3\kk,$$
$$q_\I = q_4+q_5\ii+q_6\jj+q_7\kk,$$
$q_0, q_1, q_2, q_3, q_4, q_5, q_6$ and $q_7$ are real numbers.  The multiplication of dual quaternions is also noncommutative.

The conjugate of $q = q_{st} + q_\I\epsilon$ is
$$\bar q = \bar q_{st} - q_\I\epsilon,$$
with
$$\bar q_{st} = q_0-q_1\ii-q_2\jj-q_3\kk.$$
The magnitude of $q$ is
$$|q| = |q_{st}| = \sqrt{q_0^2+q_1^2+q_2^2+q_3^2}.$$

The collections of quaternion and dual quaternion $m \times n$ matrices are denoted by ${\mathbb Q}^{m \times n}$ and ${\mathbb {DQ}}^{m \times n}$, respectively.

A dual quaternion matrix $A= (a_{ij}) \in {\mathbb {DQ}}^{m \times n}$ can be denoted as
\begin{equation} \label{e15}
A = A_{st} + A_\I\epsilon,
\end{equation}
where $A_{st}, A_\I \in {\mathbb Q}^{m \times n}$.   The transpose of $A$ is $A^\top = (a_{ji})$. The conjugate of $A$ is $\bar A = (\bar a_{ij})$.   The conjugate transpose of $A$ is $A^* = (\bar a_{ji}) = \bar A^\top$.

With these settings, in the extensions of our results to dual quaternion matrices, $\epsilon \jj$ appeared in many places of Sections 2-6 of this paper may be changed to $\epsilon$ only.

There are two major differences between dual complex matrices and dual quaternion matrices.

The first difference is that the standard parts and the infinitesimal parts of dual complex numbers are complex numbers.  The multiplication between complex numbers is commutative.   What we need to take care in Sections 2-6 is the multiplication between $\jj$ and complex numbers.   Thus, we need to use Lemma \ref{l3.1}.  The standard parts and the infinitesimal parts of dual quaternions are quaternions.  The multiplication between quaternions is not commutative.   Thus, additional care is needed in the extensions.    For example, (\ref{e4}) only can be written as
\begin{equation} \label{e7.4}
\vx_{st}\lambda_\I = A_\I\vx_{st} + A_{st}\vx_\I - \vx_\I\lambda_{st}.
\end{equation}

The second difference is that the theory of right eigenvalues and subeigenvalues of dual complex matrices is based upon eigenvalues of complex matrices, while the theory of right eigenvalues and subeigenvalues of dual quaternion matrices needs to be based upon right eigenvalues of quaternion matrices \cite{WLZZ18, Zh97}.   An $n \times n$ complex matrix has $n$ eigenvalues, while an $n \times n$ quaternion matrix has $2n$ right eigenvalues \cite{WLZZ18, Zh97}.

These two differences make the extension of Section 3 to dual quaternion matrices needs additional care.

Coming to Sections 4 and 5 of this paper, the extensions are slightly simpler, as the right eigenvalues and subeigenvalues of Hermitian matrices are real numbers, while multiplications between a real number and a dual quaternion is commutative.

Then, for Theorem 5.3 of this paper, a spectral theorem for skew-symmetric quaternion matrices is needed.  A reference for this is \cite{SW60}.

\section{Final Remarks}

In this paper, we studied right eigenvalues, right subeigenvalues and singular value decomposition of dual complex matrices.  By introducing right subeigenvalues, we fully characterize the spectral theory of dual complex Hermitian matrices.   An $n \times n$ dual complex Hermitian matrix has exactly $n$ right eigenvalues and subeigenvalues, which are all real.
The Hermitian matrix is positive semi-definite or definite if and only if all of its right eigenvalues and subeigenvalues are nonnegative or positive, respectively.  Based upon these, we present the singular value decomposition for general dual complex matrices.   We expect that these are also true for dual quaternion matrices.  %This may be the insight beauty of the nature.
We also think that these may be very useful in the future applications.   Further study on this subject would be fruitful and useful.

\bigskip

{\bf Acknowledgment}  We are thankful to Prof. Hong Yan for showing us dual numbers, dual quaternions and geometric algebra, and their applications in robotics and computer graphics.   We are grateful to Ran Gutin who pointed out an error and made detailed comments on Section 6 of the early version of this paper,
and to Chen Ling who read our draft carefully.

%{\bf Data availability statement}    The datasets generated during and/or analysed during the current study are available from the corresponding author on reasonable request.

% \vspace{100pt}

\end{document}